\documentclass[a4paper,12pt]{article}
\usepackage{caption}
\usepackage{subcaption}

\usepackage{color}
\usepackage{graphicx}
\graphicspath{{images/}}
\usepackage{latexsym}
\usepackage{amsmath}
\usepackage{amssymb}
\usepackage{enumerate}
\usepackage{pict2e}
\usepackage{theorem}
\newtheorem{theorem}{Theorem}[section]
\newtheorem{proposition}[theorem]{Proposition}

\theorembodyfont{\rmfamily}
\newtheorem{proof}{\textmd{\textit{Proof.}}}

\newtheorem{remark}[theorem]{Remark}

\newtheorem{acknowledgement}{\textmd{\textit{Acknowledgements.}}}

\makeatletter

\@addtoreset{equation}{section}
\makeatother

\newcommand{\qedd}{\hfill \Box}

\newcommand{\R}{\ensuremath{\mathbb{R}}}

\date{}
\begin{document}
\title{Finslerian indicatrices as algebraic curves and surfaces}
\author{Pipatpong Chansri, Pattrawut Chansangiam and Sorin V. Sabau\footnote{Corresponding Author}\\
} 
\maketitle

\begin{abstract}
  We show how to construct new Finsler metrics, in two and three dimensions, whose indicatrices are pedal curves or pedal surfaces of some other curves or surfaces. These Finsler metrics are generalizations of the famous slope metric, also called Matsumoto metric.

  
\end{abstract}

\let\thefootnote\relax\footnote{Balkan Journal of Geometry and Its Applications, Vol.25, No.1, 2020, pp. 19-33.}
\section{Introduction}

Finsler manifolds are natural generalizations of Riemannian manifolds in the same respect as normed spaces and Minkowski spaces are generalizations of Euclidean spaces.

In the case of the Euclidean space, or more general, of  Riemannian manifolds, the space looks uniform and isotropic, that is, the same in all direction. However, our daily experiences as well as the metrics and distances naturally appearing in applications to real life problems in Physics, Computer science, biology, etc. show that the space is not isotropic, there exists same preferred directions (see \cite{AIM}, \cite{MS}, \cite{SSS}, \cite{YS}).

To be more precise, we recall that a {\it Finsler metric} $(M,F)$ is given by specifying a Finsler norm $F:TM\to \mathbb{R}$ defined on the tangent space $(TM,M)$ of an $n$-dimensional manifold $M$. A Finsler norm has the following properties

\begin{enumerate}
\item $F$ is $C^{\infty}$ on $\widetilde{TM}:=TM\setminus\{O\}$, where $O$ is the zero section;
\item $F$ is $1$- positive homogeneous, i.e. $F(x,\lambda y)=\lambda\cdot F(x,y),\ \forall \lambda>0,\ (x,y)\in\ TM$;
\item $F$ is strongly convex, i.e. the Hessian $g_{ij}:=\frac{1}{2}\frac{\partial^2 F^2}{\partial y^i\partial y^j}(x,y)$ is positive defined for any $(x,y)\in \widetilde{TM}$.
\end{enumerate}

Observe that the fundamental function $F$ determines and it is determined by its indicatrix (the unit tangent bundle)
$\Sigma_F :=\{(x,y)\in TM:\ F(x,y)=1\}$,
which is a smooth hypersurface of $TM$. For each point $x\in M,$ we can define the indicatrix at $x$ as 
$\Sigma_x :=\{y\in TM:\ F(x,y)=1\}=\Sigma_F \cap T_x M$,
which is a smooth, closed, strictly convex hypersurface in $T_x M.$
It is therefore important to remark that to give a Finsler structure $F$ on a manifold $M$ is equivalent to giving a smooth hypersurface $\Sigma \hookrightarrow TM$ for which, the canonical projection $\pi:\Sigma \rightarrow M$ is a surjective submersion with the property that, for each $x\in M$, the $\pi$-fiber $\Sigma=\pi^{-1}(x)$ is a strictly convex hypersurface in $T_x M$ enclosing the origin. If $(M,F)$ is a Finsler manifold, then the restriction of $F$ to each tangent space $T_x M$ induce a Minkowski norm in $T_x M$. To give such a Minkowski norm is equivalent to giving the indicatrix $\Sigma_x$ at $x$. A Finsler structure on $M$ is a family of Minkowski norms $(F_x,T_x M)$ moving smoothly on the manifold.

From now on, we are going to regard Finsler and Minkowski norms as hypersurfaces in $TM$ and $T_xM$, respectively. With this image in mind, constructing examples of Finsler manifolds or Minkowski norms reduce to the effective construction of the hypersurfaces $\Sigma$ and $\Sigma_x$, respectively. Observe that the central symmetric spheres or ellipsoids give Riemannian metrics since they are all quadratic forms in the fiber coordinate $y$ of $T_xM$, hence we need to construct simple hypersurface which are not quadratic forms in $y's$.

Even though there exists already a lot of literature about Finsler manifolds and indicatrices (\cite{U}), as well as about the pedal curves (\cite{TT}) and algebraic curves in general (\cite{G}), our approach, we reconsider this topic in modern terminology, aiming to provide new insights into the theory of  Finsler spaces.

For instance, recall that the Randers and Kropina metrics are obtained by a rigid translation of the unit sphere such that the origin is enclosed by it or it is included in its boundary, respectively. 
We point out that Kropina metrics are actually conic Finsler metrics (see \cite{YS}, \cite{PCS} for details).

Another similar example of Finsler metric is the slope metric (see \cite{PCS}, \cite{M1}), where, in the $2$ dimensional case, the curve indicatrix is a {\it lima\c{c}on.}
The associated Finsler norm is written in the general form $F=\frac{\alpha^2}{\alpha-\beta}$ and called the {\it slope metric} (or a Matsumoto metric), where $\alpha=\sqrt{a_{ij}y^iy^j}$ is a Riemannian metric and $\beta=b_iy^j$ is a linear $1$-form. In \cite{PCS} we have studied the geometry of the slope metric globally induced on a surface of revolution. 

On the other hand, let us observe that a lima\c{c}on is an algebraic curve obtained as the pedal curve  of a circle
with respect to the origin. 
This insight opens a new perspective on indicatrices i.e. Finsler metrics, as algebraic curves. In the three (or higher) dimensional case it is also possible to regard indicatrices as hyper-surfaces.

In the present paper we study the following problems:
\begin{enumerate}
\item How to construct two dimensional Finsler metrics whose indicatrices are pedal curves of some algebraic curves as generalization of the slope metric and point out the convexity conditions of the pedal curves. In special we will consider the case of pedal of conics.
\item How to extend this study to the three dimensional case (and arbitrary dimensional case). This study is new in the sense that indicatrices of three dimensional slope metrics are studied for the first time. From algebraic point of view the geometry of pedal surfaces is also an interesting topic.
\end{enumerate}

Arbitrary dimension Finsler metrics whose indicatrices are pedal hypersurfaces can be studied in a similar manner, but the concrete computations can be quite messy. 

Finally, we point out that our approach is important because it illustrates and clarifies the geometrical meaning of three (and higher) dimensional slope metrics, called Matsumoto metrics in the arbitrary dimensional case. Indeed, initially, the two-dimensional slope metric was defined by Matsumoto as the Finsler metric whose indicatrix is a lima\c{c}on (see \cite{M1}). After seeing that these Finsler metrics are of the type $F=\frac{\alpha^2}{\alpha-\beta}$, where $\alpha$ is a Riemannian metric and $\beta$ a linear one form, they were simply generalized to the arbitrary dimensional case without any further considerations on the geometrical meaning. By using our pedal curves and surfaces approach one can see that the higher dimensional Matsumoto metrics are those Finsler metrics whose indicatrices are pedal hypersurfaces of spheres. 
\begin{acknowledgement}
  We are extremely grateful to Prof. H. Shimada for many useful discussions during the preparation of this manuscript.

  This research was supported by King Mongkut's Institute of Technology Ladkrabang Research Fund, grant no. KREF046201.
  \end{acknowledgement}

\section{The pedal curve of a plane algebraic curve}

Let us consider a plane algebraic curve $(\mathcal C)$ given in 
parametric form

\begin{equation}\label{parametric form 1}(\mathcal C):\quad
x=x(t),\quad 
y=y(t),
\end{equation}
then, at regular values of the parameter $t$, the tangent line to $(\mathcal C)$ is
\begin{equation}\label{tangent line eq}(\ell):\quad
\dot y(t)\cdot x -\dot x(t)\cdot y + \{\dot x(t)\cdot y (t)-x(t)\cdot \dot y(t)\}=0,
\end{equation}
 and the orthogonal line to $(\ell)$ through a point $P(x_0,y_0)$ is given by
\begin{equation*}
(\ell)^\perp:\quad y-y_0=-\frac{\dot{x}(t)}{\dot{y}(t)}(x-x_0),
\end{equation*}
where dots represent the derivative of a function of one variable with respect to $t$.

For a regular plane algebraic curve  $(\mathcal C)$, and a fixed point  $P(x_0,y_0)$, called the {\it pedal point}, the {\it pedal curve} of the curve $(\mathcal C)$ with respect to $P$ is the parameterized curve obtained by associating to the parameter $t$ the orthogonal projection $p(t)$ of $P$ onto the tangent line at $t$ (see 
\cite{G}, \cite{TT} for details on algebraic curves).
The pedal curves are considered important in geometrical optics and kinematics. 





We recall that the moving equation of the Frenet frame $(T(t),N(t))$ along $(\mathcal C)$ are given by
\begin{equation*}
\begin{cases}
\dfrac{dT}{dt}=|c'(t)|\cdot k_c\cdot N(t)\\
\ \\
\dfrac{dN}{dt}=-|c'(t)|\cdot k_c\cdot T(t),
\end{cases}
\end{equation*}
where $(T(t), N(t))$ are the unit tangent and normal vectors along $c$, respectively, $|c'(t)|=\sqrt{\dot{x}(t)^2+\dot{y}(t)^2}$ is the speed of $(\mathcal C)$ and
\begin{equation*}
k_c:=\frac{\langle c''(t),N(t)\rangle}{|c'(t)|^2}
\end{equation*}
is the curvature of the curve $(\mathcal C)$. Here $\langle\cdot,\cdot\rangle$ is the usual inner product of the Euclidean plane.

 A straighforward computation shows that 
  the pedal curve of $(\mathcal C)$ with respect to the point $P$ is given by
\begin{equation}\label{2.4}
  p(t)=\langle c-r_0,N\rangle\cdot N+r_0,
\end{equation}
where we denote by $r_0$ the position vectors of the point $P(x_0,y_0)$.

\begin{proposition}
  If $(\mathcal C)$ is a continuously differentiable closed curve in plane, then its pedal
  curve $(\mathcal P):p=p(t)$ with respect to a point $P(x_0,y_0)$ is also a continuously differentiable closed curve in plane.
\end{proposition}
\begin{proof}
From hypothesis, after some rescaling of the parameter $t$, we have
$c(0)=c(2\pi),\ \dot{c}(0)=\dot{c}(2\pi),$
and hence
$T(0)=T(\pi),\ N(0)=N(\pi).$
Using now \eqref{2.4} it follows
$p(0)=p(2\pi),$
i.e. $p$ is also periodic with the some period as $(\mathcal C)$. Moreover, $\dot{p}(0)=2\pi$, where by derivation of \eqref{2.4} we get
\begin{equation}\label{2.5}
\dot{p}(t)=k_c\cdot |c'|\cdot \left[\langle r_0-c,T\rangle\cdot N+\langle r_0-c,N \rangle\cdot T \right].
\end{equation}

$\qedd$
\end{proof}

We will ask now the question if the pedal curve $p(t)$ goes though the origin $O(0,0)$ of $\mathbb{R}^2$. This is equivalent to asking if the vectorial equation
\begin{equation*}
p(t)=\langle c,N \rangle\cdot N+\langle r_0,T\rangle\cdot T=(0,0)^t,\ \text{where $t$ denote the transposed matrix,}
\end{equation*}
has solution. Since $N$ and $T$ linearly independent, this equation is equivalent to the system of equation,
\begin{equation}\label{2.6}
\langle c,N\rangle=0,\quad 
\langle r_0,N\rangle=0.
\end{equation}
We consider
\begin{enumerate}
\item[Case 1.] Assume the pedal point is origin, i.e. $r_0=(0,0)$. In this case we get only the equation
\begin{equation*}
\langle c,N\rangle=-\left| \begin{array}{cc} c_1 & c_2\\ \dot{c}_1 & \dot{c}_2
\end{array}  \right|=0,
\end{equation*}
and observe that for a continuous differentiable curve this is possible if and only if
$(\mathcal C)$ passes through origin.

Hence, in this case $(\mathcal P)$ passes through origin if and only if $(\mathcal C)$ passes through origin.
\item[Case 2.] Assume the pedal point $P$ is not the origin, i.e.  $x_0\neq 0$ or $y_0 \neq 0$. Then we consider further the cases:
\begin{enumerate}
\item[2.1] The curve $(\mathcal C)$ passes through origin, i.e. there exists $t_0$ to such that $c(t_0)=(0,0)$, $\dot{c}(t_0)\neq (0,0)$. In this case we obtain $\langle r_0,T(t_0)\rangle =0$, i.e. $r_0$ and $T(t_0)$ are orthogonal.
\item[2.2] The curve $(\mathcal C)$ do not pass through the origin, i.e. both conditions in \eqref{2.6} must be simultaneously verified, but this is impossible. Indeed, since $(\mathcal C)$ is continuously differentiable curve, $c(t)$ cannot be collinear to $T(t)$, nor $r_0$ can be always orthogonal to $T$ since $(\mathcal C)$ is a closed curve.
\end{enumerate}
\end{enumerate}
We conclude:
\begin{theorem}\label{theorem 2.3}
Let $\mathcal C)$ be a plane algebraic curve with parametric equation \eqref{parametric form 1} and $P(x_0,y_0)$ a point in $\mathbb{R}^2$, $P\notin (\mathcal C)$.
\begin{enumerate}
\item[$\bullet$]If $P$ is the origin, then $(\mathcal P)$ passes through origin if and only if $(\mathcal C)$ also passes through origin.
\item[$\bullet$] If $P$ is not the origin, and $(\mathcal C)$ passes through the origin then $(\mathcal P)$ passes through origin if and only if $r_0$ and $T(t_0)$ are orthogonal, where $t_0$ is the value of the parameter $t$ such that $c(t_0)=(0,0)$.
\item[$\bullet$] If $P$ is not the origin, and $(\mathcal C)$ does not pass through origin, then $(\mathcal P)$ does not pass origin either.
\end{enumerate}
\end{theorem}
Next, we will study the convexity condition of the pedal curve $(\mathcal P)$ in \eqref{2.4}. We recall that a curve $p=p(t)$ is strongly convex if and only if
\begin{equation}\label{2.7}
\frac{\dot{p}(t)\times \ddot{p}(t)}{p(t)\times\dot{p}(t)}>0,
\end{equation}
where the cross product of two vector $u=(a,b),\ v=(c,d)$ is given by $u\times v=ad-bc$, (see for instance \cite{BCS}, page 88).
Observe that this condition is independent on the parameterization of $p(t)$. A straightforward computation gives
\begin{equation*}
\begin{split}
\frac{\dot{p}\times\ddot{p}}{p\times\dot{p}}&=\frac{\left|\begin{array}{cc} v_1 & v_2\\ u_1 & u_2
\end{array}  \right|}{\left|\begin{array}{cc} v_0 & v_1\\ u_0 & u_1
\end{array}  \right|}
=\frac{k_c|\dot{c}|^2\left|\begin{array}{cc} \langle r_0-c,T\rangle & -1+2k_c\langle r_0-c,N\rangle\\ \langle r_0-c,N\rangle & -2k_c\langle r_0-c,T\rangle
\end{array}  \right|}{\left|\begin{array}{cc} \langle c,N\rangle & \langle r_0-c,N\rangle\\ \langle r_0,T\rangle & \langle r_0-c,T\rangle
\end{array}  \right|}\\
&=\frac{k_c|\dot{c}|^2\left[-2k_c\langle r_0-c,T \rangle^2+\langle r_0-c,N \rangle-2k_c\langle r_0-c,N \rangle^2 \right]}{\langle c,N\rangle\langle r_0-c,N\rangle -\langle r_0,T\rangle\langle r_0-c,T\rangle},
\end{split}
\end{equation*}
where we have used
$\dot{p}(t)=-k_c\cdot |\dot{c}|^2\{\langle c,T\rangle\cdot N+ \langle c,N\rangle\cdot T \}.$
Hence, we obtain
\begin{theorem}
The pedal curve $(\mathcal P)$ is strongly convex if and only if
\begin{equation}\label{2.10}
\begin{split}
k_p(t):=\frac{k_c\{-2k_c\langle r_0-c,T \rangle^2+\langle r_0-c,N \rangle-2k_c\langle r_0-c,N \rangle^2\}}{\langle c,N\rangle\langle r_0,N\rangle+\langle c,T\rangle\langle r_0,T\rangle -\langle c,N\rangle^2-\langle r_0,T\rangle^2}>0,
\end{split}
\end{equation}
for all $t\in [0,2\pi).$
\end{theorem}
\begin{remark}$ $
  In the case when the pedal point $P$ is the origin, formula \eqref{2.10} simplifies to
  \begin{equation}\label{simplif form}
    k_c\left[2k_c\langle c,T\rangle^2+\langle c,N\rangle+2k_c\langle c,N\rangle^2\right]>0.
  \end{equation}
  
Moreover, observe that the position vector $c(t)$ can be decomposed in the orthonormal basis $\{T, N\}$ as
$
c(t)=\langle c(t),T(t)\rangle T(t)+\langle c(t),N(t)\rangle N(t),
$
hence
$
\langle c(t),c(t)\rangle= \langle c(t),T(t)\rangle^2+\langle c(t),N(t)\rangle^2.
$
It is now easy to see now that the formula \eqref{simplif form} is equivalent to 
$k_c\left[2k_c|c|^2+\langle c,N\rangle\right]>0.$


\end{remark}
\section{Some remarkable pedal curves and their corresponding Finsler metrics}
\subsection{The slope metric whose indicatrix is a lima\c{c}on}
\subsubsection{ The pedal curve}

%

It is easy to see that the pedal curve of the circle with center $(0,k)$ and radius $a$ with respect to the origin of $\mathbb{R}^2$ is a lima\c{c}on.

Indeed, the curve $(\mathcal C)$ is the circle $(x-k)^2+y^2=a^2$, then 
the equation \eqref{2.4} gives the pedal curve
\begin{equation*}
\begin{cases}
p_1(t)=(a+k\cos t)\cdot \cot t\\
p_2(t)=(a+k\cos t)\cdot \sin t.
\end{cases}
\end{equation*}
This is equivalent with the polar equation
$r=a+k\cdot\cos \theta,$
where $(r,\theta)$ are the polar coordinates in $\mathbb{R}^2$, or the implicit equation
\begin{equation}\label{limacon implicit}
(x^2+y^2-kx)^2=a^2(x^2+y^2).
\end{equation}
Observe that the pedal pedal curve $(\mathcal P)$ is not passing through origin (Theorem \ref{theorem 2.3}). Moreover, the curvature of the lima\c{c}on reads now
\begin{equation*}
k_p(t)=\frac{a^2
  +3ak\cos t +2k^2}{a^2+2ak\cos t+k^2},
\end{equation*}
and since $a^2+2ak\cos t+k^2\geq a^2-2ak+k^2=(a-k)^2>0$, for $a\neq k$, 
the condition $k_p(t)>0$ is therefore equivalent to
$a^2
  +3ak\cos t +2k^2>0.$
Observe that the minimum of this expression is obtained for $\cos t=-1,$ hence
$a^2
+3ak\cos t +2k^2>a^2
-3ak +2k^2=(a-k)(a-2k)>0,$
for  $a>2k$.

This is the convexity condition for the pedal $p(t)$ in this case (see Figure \ref{fig1}). 

	\begin{figure}[ht]
    \centering
    \begin{subfigure}[b]{0.35\textwidth}
    \centering
        \includegraphics[width=\textwidth]{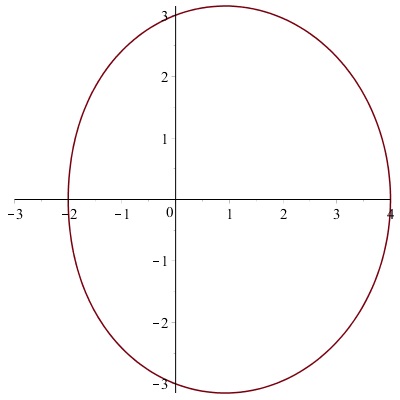}
    \end{subfigure}\quad\quad\quad\quad
          \begin{subfigure}[b]{0.35\textwidth}
          \centering
        \includegraphics[width=\textwidth]{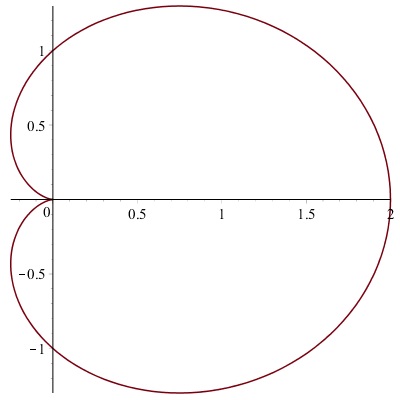}
        \end{subfigure}
    \caption{A convex lima\c{c}on curve for $a=3$, $k=1$ (left) and  a non-convex  lima\c{c}on curve for $a=1$, $k=1$ (right).}\label{fig1}
\end{figure}

\subsubsection{The Finsler metric}
The Finsler metric whose indicatrix is given as a curve in each tangent space can easily determined. 
Indeed, observe that the lima\c{c}on implicit equation \eqref{limacon implicit} is equivalent to
$$
\frac{x^2+y^2}{a\sqrt{x^2+y^2}+kx}=1,
$$
hence the corresponding Minkowski norm in $\mathbb{R}^2$ is
\begin{equation*}
F(x,y)=\frac{x^2+y^2}{a\sqrt{x^2+y^2}+kx},
\end{equation*}
that is a Minkowski slope metric (see \cite{M1}, this approach is sometimes called the Okubo's method). By smoothly moving this Minkowski norm on a $2$- dimensional smooth manifold $M$ we get the usual slope metric on $M$
$F=\frac{\alpha^2}{\alpha-\beta},$
where $\alpha$ is a Riemanninan metric $M$ and $\beta$ a linear $1$-form (see our recent paper \cite{PCS} for a study of the slope metric on a surface of revolution).

In conclusion (\cite{BCS}, \cite{M1}, \cite{PCS}):
\begin{proposition}\label{slope metric convex}
  The Finsler metric on a two dimensional manifold $M$ whose indicatrix is given by the pedal curve of a circle
  $(x-k)^2+y^2=a^2$ with origin as pedal point is a slope type metric $F=\frac{\alpha^2}{\alpha-\beta}$. This Finsler metric is strongly convex for $a>2k$.
  \end{proposition}


  \begin{remark}
    We observe that a similar result can be obtained when $(\mathcal C)$ is the unit circle and $P(a,0)$ a point on the $x$-axis. By a similar computation as in the case above, we obtain the pedal curve parametric equations
    \begin{equation}\label{2.9.1}
\begin{cases}
p_1(t)= (1-a\cos t)\cdot \cos t + a \\
p_2(t)=(1-a\cos t)\cdot \sin t.
\end{cases}
\end{equation}

Observe that this curve can be regarded as a lima\c{c}on with parameters $(-a,1)$ with center translated from origin to $(a,1)$. The convexity condition reads $-\frac{1}{2}<a<\frac{1}{2}$.

We obtain that the Finsler metric on a two dimensional manifold $M$ whose indicatrix is given by the pedal curve of a unit circle
  $x^2+y^2=1$ with pedal point $(a,0)$ is of type \begin{equation*}
F=\frac{\alpha^2}{\beta_1-\sqrt{(\alpha-\beta)^2+\beta_2^2}},
\end{equation*}
where $\alpha$ is a Riemannian metric and $\beta_1,\beta_2$ are the linear form in $TM$. This Finsler metric is strongly convex for $a\in\left( -\frac{1}{2},\frac{1}{2}\right)$.
    
    \end{remark}

\subsection{The pedal curve of an ellipse}\label{sec:3.1}
\subsubsection{The pedal curve}
\begin{remark}[Motivation]
Let us recall that two polynomials $P,Q$ in $x,y$ with real coefficients are equivalent if there exists a non zero $\lambda\in\mathbb{R}$ such that $P=\lambda\cdot Q$. This is an equivalence relation on the set of polynomials and an equivalence class is called an {\it affine plane curve}. Moreover, two affine curves
$(c_1):\ f(x,y)=0$,  
$(c_2):\ g(x,y)=0$ 
are called {\it affinely equivalent} if there exists an affine map $\phi$ on $\mathbb{R}^2$ and a scalar $\lambda\neq 0$ such that $g(x,y)=\lambda\cdot f(\phi(x,y))$. Since the set of affine maps on $\mathbb{R}^2$ is a group $(Aff(2),0)$, with the operation of composition, affine equivalence defines an equivalence relation for plane curves in $\mathbb{R}^2$. Observe that the degree $d$ of a curves is an affine invariant. Clearly $d=1$ gives straight lines, so they are not interesting for us.

The next simple case is $d=2$, i.e. conic, the circle being affinely equivalent to real ellipse, which is the only closed and convex conic. 
\end{remark}

It is therefore naturally to consider the general case when the curve $(\mathcal C)$ is an ellipse,
i.e.
\begin{equation*}(\mathcal C):\ 
x=k+a\cos t,\quad  
y=b\sin t
,\quad \ k>0,\ b>0,\ a>0,\ a\neq b,
\end{equation*}
and $P(x_0,y_0)$ an
arbitrary point.


The pedal curve of $(\mathcal C)$ with respect to the pedal point $P(x_0,y_0)$ is
\begin{equation*}
p(t)=\frac{1}{|c'|^2}\left\{b(k\cos t+a)\left( \begin{array}{c} b\cos t\\ a\sin t
\end{array}  \right)+(-x_0\cdot a\sin t+y_0 b\cdot\cos t)\left( \begin{array}{c} -a\sin t\\ b\cos t
\end{array}\right) \right\}.
\end{equation*}

For the sake of simplicity, we consider the case when $P\equiv O$ is the origin. In this case, the pedal curve has the parametric equations
\begin{equation}\label{curve3}
\begin{cases}
p_1(t)=\frac{1}{|c'|^2}b^2(k\cos t+a) \cos t\\
p_2(t)=\frac{1}{|c'|^2}ab(k\cos t+a) \sin t,
\end{cases}
\end{equation}
and from here it follows the implicit equation
$a^2x^2+b^2y^2=(x^2+y^2-kx)^2.$


Recall that a curve $p=p(t)$ is convex if and only if $k_c[2k_c|c|^2+\langle c,N\rangle]>0$, so we compute this equation and get the condition
\begin{equation*}
				\begin{split}
				-a^3+2ak^2+2ab^2+(a^2-b^2)\cos^2 t(k\cos t+3a)+3ka^2\cos t>0.
					\end{split}
				\end{equation*}

                                Again for simplicity we can consider $a>b$, then we  have to check that
$$3 a^2k \cos t -a^3+2 a k^2+2 a b^2>-3 a^2k  -a^3+2 a k^2+2 a b^2>0,$$
then strongly convexity reads (see Figure \ref{fig2})
\begin{equation}\label{ellipse convex condition}
a>b>\frac{1}{
\sqrt{2}}\sqrt{3ak+a^2-2k^2}.
\end{equation}

\subsubsection{The Finsler metric}

If we apply Okubo's method we obtain the Minkowski norm
\begin{equation}\label{Mink for ellipse}
F=\frac{x^2+y^2}{\sqrt{a^2x^2+b^2y^2}+kx},
\end{equation}
with the strongly convexity condition $a>b>\frac{1}{\sqrt{2}}\sqrt{3ak+a^2-2k^2}$. Observe that this gives the Finsler metric on $M$
\begin{equation}\label{Finsler for ellipse}
F=\frac{\alpha_1^2}{\alpha_2-\beta},
\end{equation}
where $\alpha_1,\alpha_2$ are two different Riemannian metrics. In the case $\alpha_1=\alpha_2$ we obtain the usual slope metric.
	\begin{figure}[ht]
    \centering
    \begin{subfigure}[b]{0.35\textwidth}
    \centering
        \includegraphics[width=\textwidth]{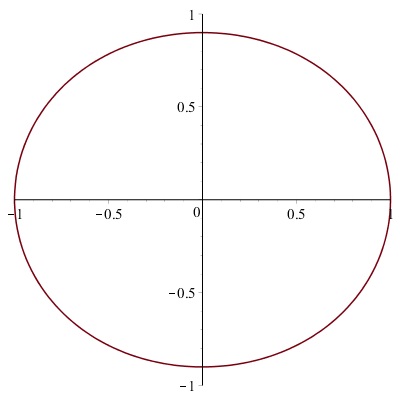}
    \end{subfigure}\quad\quad\quad\quad
          \begin{subfigure}[b]{0.35\textwidth}
          \centering
        \includegraphics[width=\textwidth]{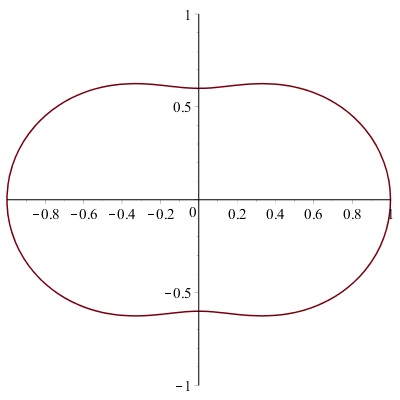}
        \end{subfigure}
    \caption{The convex curve in \eqref{curve3} for $a=10,\ b=9,\ k=2$ (left), and the non-convex case for $a=10,\ b=6,\ k=2$ (right). }\label{fig2}
\end{figure}

 We summarize
  \begin{theorem}\label{ellipse indicatrix}
  The Finsler metric on a two dimensional manifold $M$ whose indicatrix is given by the pedal curve of the ellipse
  $\left(\frac{x-k}{a}\right)^2+\left(\frac{y}{b}\right)^2=1$ with origin as pedal point is of type
$
F=\frac{\alpha_1^2}{\alpha_2-\beta},
$
where $\alpha_1,\alpha_2$ are two different Riemannian metrics and $\beta$ is a linear form in $TM$. This Finsler metric is strongly convex for $a>b>\frac{1}{
\sqrt{2}}\sqrt{3ak+a^2-2k^2}.$
\end{theorem}

\begin{remark}
Observe that if put $a=b$ in theorem \eqref{ellipse indicatrix} we obtained the Finsler metric
$F=\frac{x^2+y^2}{a\sqrt{x^2+y^2}+kx},$
with strongly convexity condition $a>\dfrac{1}{\sqrt{2}}\sqrt{3ak+a^2-2k^2}$, or, equivalently 
\begin{equation}
2a^2>3ak+a^2-2k^2\ \Rightarrow\
(a-k)(a-2k)>0.
\end{equation}
Therefore we obtain $a>2k$, that is same result as in Proposition \ref{slope metric convex}.
\end{remark}

\section{The pedal surface}
We are going to extend our considerations from curves of surfaces. Instead of the curve $(\mathcal C)$, we are going to consider a smooth surface $\mathcal S\hookrightarrow \mathbb{R}^3$ embedded in $\mathbb{R}^3$ with parametric equations
\begin{equation}\label{3.1}
(\mathcal S):
x=x(u,v),\quad 
y=y(u,v), \quad 
z=z(u,v),
\end{equation}
and observe that, at any regular vector $(u,v)$ of the parameters, the tangent plane to $(\mathcal S)$ at $(x(u,v),y(u,v),z(u,v))\in \mathcal S$ is given by 
\begin{equation*}
\begin{split}
(\pi):\quad& \frac{\partial (y,z)}{\partial (u,v)}(x-x(u,v))+\frac{\partial (z,x)}{\partial (u,v)}(y-y(u,v))+\frac{\partial (x,y)}{\partial (u,v)}(z-z(u,v))=0,
\end{split}
\end{equation*}
where
\begin{equation*}\quad \frac{\partial (y,z)}{\partial (u,v)}=
				\left| \begin{array}{cc}
					\frac{\partial y}{\partial u} & \frac{\partial y}{\partial v}\\
					\frac{\partial z}{\partial u}& \frac{\partial z}{\partial v}
				\end{array} \right|=\left| \begin{array}{cc}
					y_u & y_v\\
					z_u& z_v
				\end{array} \right|
				\end{equation*}
and so on. The normal to $(\pi)$ at a point $(u,v)$ is given by
\begin{equation*}
\begin{split}
(\pi^{\perp}):\quad \frac{x}{\frac{\partial (y,z)}{\partial (u,v)}}=\frac{y}{\frac{\partial (z,x)}{\partial (u,v)}}=\frac{z}{\frac{\partial (x,y)}{\partial (u,v)}}.
\end{split}
\end{equation*}
Let $(\mathcal S)$ be a regular surface parameterized on in \eqref{3.1} and let $P(x_0,y_0,z_0)$ be a fixed point, the {\it pedal point}. Then the {\it pedal surface} of the surface $(\mathcal S)$ with respect to the point $P$ is the parameterized surface obtained by associating to the parameter $(u,v)$ the orthogonal projection $p(u,v)$ of $P$ onto the tangent plane $(\pi)$ at $\mathcal S (u,v)$.

The tangent plane $(\pi)$ is generated by the vectors
$  \mathcal S_u=\left(
x_u,y_u,z_u
  \right)^t$, and 
$ \mathcal S_v=
\left(
x_v, y_v, z_v
\right)^t$, 
while the unit normal vector to $(\mathcal S)$ is
$$N=\frac{\mathcal S_u\times\mathcal S_v}{||\mathcal S_u\times\mathcal S_v||}=\frac{1}{||\mathcal S_u\times\mathcal S_v||}\left( \frac{\partial (y,z)}{\partial (u,v)},\frac{\partial (z,x)}{\partial (u,v)},\frac{\partial (x,y)}{\partial (u,v)} \right)^t.$$
Similarly with the plane curve's case, a straightforward computation shows
that the pedal surface of the surface $(\mathcal S)$ with respect to the point $P$ is given by 
\begin{equation}\label{pedal general formula}
  p(u,v)=\langle\mathcal S-r_0,N\rangle\cdot N+r_0.
  \end{equation}


 The convexity condition of the pedal surface is given by the condition $K>0$, there $K$ is the Gauss curvature, that is,
\begin{equation}\label{convex pedal surface}
\left|\begin{array}{cc} \langle p_{uu},p_u\times p_v\rangle & \langle p_{uv},p_u\times p_v\rangle\\\langle p_{uv},p_u\times p_v\rangle& \langle p_{vv},p_u\times p_v\rangle
\end{array} \right|>0.
\end{equation}
The formula can be quite complicate in the general case, but we will consider some examples.

\begin{remark}
  In the same way we can define the pedal hypersurface of an $n$-dimensional surface $\mathcal S \hookrightarrow \R^{n+1}$. The formula \eqref{pedal general formula}
  is clearly true for any dimensions, but the sectional curvature computations and the determination of the strongly convexity condition becomes more difficult. Nevertheless, in the case of the $n$-sphere the computations are quite straightforward as we shall see. 
  \end{remark}

\section{Some remarkable  pedal surfaces and the 
  corresponding Finsler metrics}
\subsection{The pedal surface of a 2-sphere}
\subsubsection{The pedal surface}
The easiest case is in the case where $(\mathcal S)$ in the $2$-sphere $\mathbb{S}^2\hookrightarrow \mathbb{R}^3$ with center $(k,0,0)$ and radius $r$, i.e.
\begin{equation*}(\mathcal S):\
x=k+r\sin v\cos u,\quad 
y=r\sin v\sin u,\quad 
z=r\cos v
,\quad \ k>0,\ r>0.
\end{equation*}
Then the exterior oriented unit normal vector is
$$N=\frac{\mathcal S_v\times \mathcal S_u}{||\mathcal S_v\times \mathcal S_u||}=
\left( \cos u\sin v, \sin u\sin v, \cos v\right)^t,
$$
and hence the pedal surface of the $2$-sphere $(\mathcal S)$ center at $(k,0,0)$ with respect to the pedal point $P\equiv O$ (origin of $\mathbb{R}^3$) is
\begin{equation*}
p(u,v):
\begin{cases}
x(u,v)=\sin v(r+k\cos u\sin v)\cos u\\
y(u,v)=\sin v(r+k\cos u\sin v)\sin u\\
z(u,v)=\cos v(r+k\cos u\sin v).
\end{cases}
\end{equation*}

The implicit equation of $p(u,v)$ can be written in the form $f(x,y,z)=0$ where
\begin{equation*}
f(x,y,z)=x^2+y^2+z^2-r\sqrt{x^2+y^2+z^2}-kx.
\end{equation*}

This surface can be called the {\it lima\c{c}on surface}, or the two dimensional lima\c{c}on.

We recall that a surface is called strongly convex  when $LN-M^2>0$, where $L=\langle p_{uu},V \rangle,\ N=\langle p_{vv},V \rangle$ and $M=\langle p_{uv},V \rangle$. Then, the unit normal vector  is given by $V:={\nabla f} \slash {||\nabla f||}$, hence the strongly convexity condition reads $\langle p_{uu},\nabla f \rangle\langle p_{vv},\nabla f \rangle-\langle p_{uv},\nabla f \rangle^2>0$, and a straightforward computation gives

\begin{equation*}\nabla f:\left( \dfrac{\partial f}{\partial x},\dfrac{\partial f}{\partial y},\dfrac{\partial f}{\partial z}\right)^{t}=
\left(\begin{array}{l} (2Ak+r)A-k \\(2Ak+r)\sin u\sin v \\(2Ak+r)\cos v
\end{array} \right),
\end{equation*}
where $A:=\cos u\sin v$. 

Moreover, we have
%
\begin{equation*}\begin{split}
\langle p_{uu},\nabla f \rangle=&\left[(-4kA^2-rA+2k\sin^2 v)(2A^2k+rA-k)-(4kA +r)(2kA +r)\sin^2 v \sin^2 u\right.\\
&\left.-kA(2Ak+r)\cos^2 v\right]\\
\langle p_{vv},\nabla f \rangle=&\left[ (-4kA^2-rA+2k\cos^2 u)(2A^2k+rA-k)+2kA(2Ak+r)\sin^2 u\right.\\
&\left.-(4kA +r)(2kA +r)(1-A^2)\right]\\
\langle p_{uv},\nabla f \rangle=&2Ak^2\cos v\sin u.
\end{split}
\end{equation*}
The strongly convexity condition reads now
\begin{equation*}
\langle p_{uu},\nabla f \rangle\langle p_{vv},\nabla f \rangle-\langle p_{uv},\nabla f \rangle^2=\sin^2 v(3Akr+2k^2+r^2)(Ak+r)(2Ak+r)>0.
\end{equation*}

Taking into account that $-1\leq A\leq 1$, and $r,k>0$ it results
that the pedal surface is strongly convex for $r>2k$. 
This condition is consistent with the  condition obtained in the case of the pedal curve of the circle (see Figure \ref{fig3}).
\subsubsection{The Finsler metric}

The Minkowski metric associated  can be easily obtained by Okubo's method
$$F=\frac{x^2+y^2+z^2}{r\sqrt{x^2+y^2+z^2}+kx},$$
that is a slope metric  $F=\frac{\alpha^2}{\alpha-\beta}$ on the surface $M$.

\begin{figure}[ht]
    \centering
    \begin{subfigure}[b]{0.4\textwidth}
          \centering
        \includegraphics[width=\textwidth]{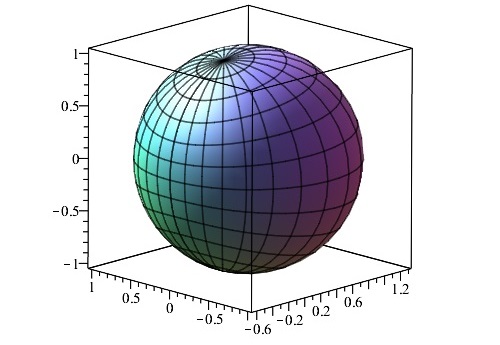}
        \end{subfigure}\quad\quad\quad
    \begin{subfigure}[b]{0.3\textwidth}
    \centering
        \includegraphics[width=\textwidth]{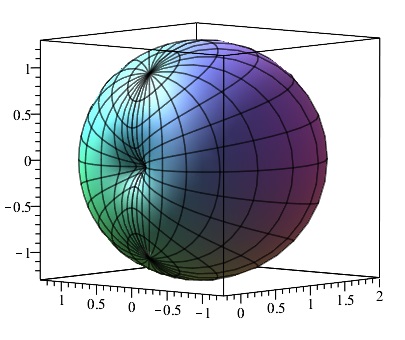}
    \end{subfigure}\quad
    \caption{The convex pedal surface of the sphere, with pedal point in origin, for $k=\frac{1}{3},\ r=1$ (left), and the non-convex case for $k=1,\ r=1$ (right).}\label{fig3}
\end{figure}

 By smoothly moving this Minkowski norm on a $3$- dimensional smooth manifold $M$ we get the usual slope metric on $M$
$F=\frac{\alpha^2}{\alpha-\beta},$
where $\alpha$ is a Riemanninan metric $M$ and $\beta$ a linear $1$-form.

In conclusion we get:
\begin{theorem}
  The Finsler metric on a three dimensional manifold $M$ whose indicatrix is given by the pedal surface of a sphere
  $(x-k)^2+y^2+z^2=r^2$ with origin as pedal point is a slope type metric $F=\frac{\alpha^2}{\alpha-\beta}$, where $\alpha$ is a Riemannian metric on $M$ and $\beta$ a linear one form on $TM$. This Finsler metric is strongly convex for $r>2k$.
\end{theorem}

\begin{remark}
  Without giving here the concrete computations, a quick look at the formulas above show that the same is true for the arbitrary dimensional case as well. We only formulate here without proof the following
  \begin{theorem}
  The Finsler metric on an $n\geq 2$-dimensional manifold $M$ whose indicatrix is given by the pedal hypersurface of an $n-1$-sphere
  $(x_1-k)^2+x_2^2+\dots+x_n^2=r^2$ with origin as pedal point is a slope type metric $F=\frac{\alpha^2}{\alpha-\beta}$, where $\alpha$ is a Riemannian metric on $M$ and $\beta$ a linear one form on $TM$. This Finsler metric is strongly convex for $r>2k$.
\end{theorem}
  \end{remark}

  \subsection{The Pedal surface of an ellipsoid}
  As a generalization of the Section \ref{sec:3.1}, we can consider the Finsler metric on a surface $M$ whose indicatrix is the pedal surface of an ellipsoid. The computations are quite long, so we give only some ideas of the construction in this section.
  
        \subsubsection{The pedal surface}
The parametric equations of an ellipsoid can be written as 
\begin{equation*}(\mathcal S):
x=k+a\sin v\cos u,\quad 
y=b\sin v\sin u,\quad 
z=c\cos v
,\quad \ k>0,\ a>0,\ b>0,\ c>0.
\end{equation*}

The exterior oriented unit normal reads
\begin{equation*}
N=\frac{\mathcal S_v\times \mathcal S_u}{||\mathcal S_v\times \mathcal S_u||}=\frac{1}{||\mathcal S_v\times \mathcal S_u||}
\left(\begin{array}{c} bc\sin v\cos u\\ ac\sin v\sin u\\ ab\cos v
\end{array} \right),
\end{equation*}
where $||\mathcal S_v\times \mathcal S_u||=\sqrt{c^2b^2\sin^2 v\cos^2 u+c^2a^2\sin^2 v\sin^2 u+a^2b^2\cos^2 v}$.

By formula \eqref{pedal general formula}, 
the pedal surface of $\mathcal{S}$ with respect the pedal point $P(x_0,y_0,z_0)$ is
\begin{equation}\label{Ellipsoid Pedal general}
\begin{split}
p(u,v)=&\sin v\frac{\left(bc(k-x_0)\sin v\cos u+a(-cy_0\sin v\sin u -bz_0\cos v+bc)\right)}{||\mathcal S_v\times \mathcal S_u||^2}
\left(\begin{array}{c} bc\sin v\cos u\\ ac\sin v\sin u\\ ab\cos v
\end{array} \right)\\
&+(x_0,y_0,z_0)^t.
\end{split}
\end{equation}


This general case
implies some long computations, but 
we can again consider the case when $P\equiv O$ is the origin.
In this case we obtain
\begin{equation}\label{Ellipsoid Pedal}
p(u,v):\quad
\begin{cases}
x(u,v)=\dfrac{b^2c^2(k\cos u\sin v +a)}{||\mathcal S_v\times \mathcal S_u||^2}\sin v\cos u\\
y(u,v)=\dfrac{abc^2(k\cos u\sin v +a)}{||\mathcal S_v\times \mathcal S_u||^2}\sin v\sin u\\
z(u,v)=\dfrac{ab^2c(k\cos u\sin v +a)}{||\mathcal S_v\times \mathcal S_u||^2}\cos v.
\end{cases}
\end{equation}

From \eqref{ellipse convex condition} we can see that
\begin{equation}\label{ellipsoid equation1}
a^2x^2+b^2y^2+c^2z^2=\frac{a^2b^4c^4(k\cos u\sin v +a)^2}{||\mathcal S_v\times \mathcal S_u||^4}.
\end{equation}

On the other hand,
\begin{equation}\label{ellipsoid equation2}
x^2+y^2+z^2-kx=\dfrac{ab^2c^2(k\cos u\sin v +a)}{||\mathcal S_v\times \mathcal S_u||^2},
\end{equation}
and by comparing \eqref{ellipsoid equation1} and \eqref{ellipsoid equation2} we get the implicit equation of the pedal surface:
\begin{equation}\label{implicit pedal ellipsoid}
(x^2+y^2+z^2-kx)^2=a^2x^2+b^2y^2+c^2z^2.
\end{equation}

Finding an explicit form of general conditions for the strongly convexity of the ellipsoid pedal involves some long computations, that we omit. Some numerical simulations show that for instance, in the case $k=1\slash 3$, $a=b=2$, $c=\sqrt{6}$, we indeed obtain a strongly convex surface, see Figure \ref{fig4}.


\begin{figure}[ht]
    \centering
    \begin{subfigure}[b]{0.3\textwidth}
    \centering
        \includegraphics[width=\textwidth]{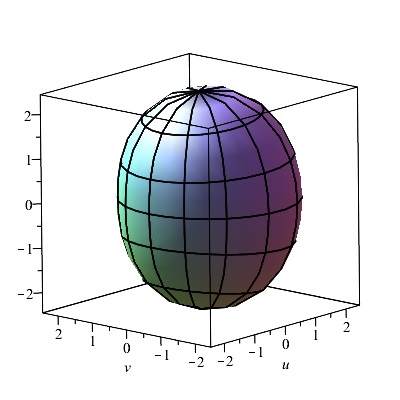}
   \end{subfigure}\quad\quad\quad\quad
          \begin{subfigure}[b]{0.3\textwidth}
          \centering
        \includegraphics[width=\textwidth]{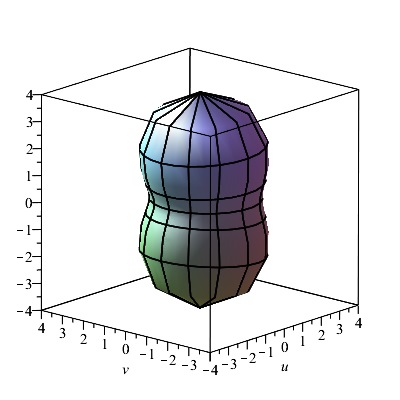}
        \end{subfigure}
    \caption{The convex pedal surface of an ellipsoid, with pedal point at origin, for  $k=1\slash 3$, $a=b=2$, $c=\sqrt{6}$ and the non-convex case for $k=1 \slash 3$, $a=b=2$, $c=4$.}\label{fig4}
\end{figure}
%


\subsubsection{The Finsler metric}

Applying Okubo's method to \eqref{implicit pedal ellipsoid} we obtain the Minkowski norm
$F=\frac{x^2+y^2+z^2}{\sqrt{a^2x^2+b^2y^2+c^2z^2}+kx}$
that is clearly the generalization of \eqref{Mink for ellipse}. 

The  Finsler metric corresponding to this indicatrix surface is of the type \eqref{Finsler for ellipse}, where $\alpha_1,\alpha_2$ are two different Riemannian metrics. This the generalization of the discussion in Section \ref{sec:3.1}.

 We can summarize
  \begin{theorem}\label{ellipsoid indicatrix}
  The Finsler metric on a three dimensional manifold $M$ whose indicatrix is given by the pedal surface of the ellipsoid
  $\left(\frac{x-k}{a}\right)^2+\left(\frac{y}{b}\right)^2+\left(\frac{z}{c}\right)^2=1$ with origin as pedal point is of type
$
F=\frac{\alpha_1^2}{\alpha_2-\beta},
$
where $\alpha_1,\alpha_2$ are two different Riemannian metrics and $\beta$ is a linear form in $TM$. This Finsler metric is strongly convex subject to some conditions for $a$, $b$, $c$ and $k$.
\end{theorem}

\begin{remark}
  Similarly with the sphere case, without giving here the concrete computations, one can easily see that the same formulas are true for the arbitrary dimensional case as well.
  
  The Finsler metric on an $n$-dimensional manifold $M$ whose indicatrix is given by the pedal hypersurface of an ellipsoid 
    with origin as pedal point is a slope type metric $F=\frac{\alpha_1^2}{\alpha_2-\beta}$, where $\alpha_1,\alpha_2$ are two different Riemannian metrics  on $M$ and $\beta$ a linear one form on $TM$. This Finsler metric is strongly convex for some further conditions on the constants giving the axes of the ellipsoid and the coordinates of its center.
  \end{remark}


\end{document}